\documentclass[12pt]{amsart}
\usepackage{amssymb,amsmath,amsthm,latexsym}
 \newtheorem{theorem}{Theorem}[section]
 \newtheorem{corollary}[theorem]{Corollary}
 \newtheorem{lemma}[theorem]{Lemma}
 \newtheorem{proposition}[theorem]{Proposition}
 \theoremstyle{definition}
 \newtheorem{definition}[theorem]{Definition}
 \theoremstyle{remark}
 \newtheorem{remark}[theorem]{Remark}
 
 \numberwithin{equation}{section}
 \newtheorem{notation}[theorem]{Notation}

 \def\BMO{{\rm BMO}}

\def\CMO{{\rm CMO}}

\def\cal{\mathcal}

\def\R{{\mathbb R}}

\def\supp{{\rm supp}}

\def \diam {{\rm diam}}

\def \dist {{\rm dist}}

\def \rdist {{\rm \, rdist}}
\def \ec {{\rm ecc}}
\def \ecc {{\rm ecc}}

\newcommand\smland{\raise.2ex\hbox{$\scriptstyle\land$}\hspace{.01cm}}
\newcommand\smlor{\raise.1ex\hbox{$\scriptstyle\lor$}}

\begin{document}
%
%
%
%
%
%
%
%
%
\title[Estimates for compact Calder\'on-Zygmund  operators]
{Off-diagonal and pointwise estimates \\ for compact Calder\'on-Zygmund operators}
\author{Paco Villarroya}

\address{Centre for Mathematical Sciences\\ University of Lund\\ Lund\\ Sweden}

\email{paco.villarroya@maths.lth.se}

\thanks{This work was completed with the support of Spanish project MTM2011-23164.}


\subjclass
{Primary 42B20, 42C40; Secondary 47B07, 47G10}

\keywords{Singular integral, Calder\'on-Zygmund operator, compact operator, off-diagonal estimates}

\date{today}

\begin{abstract}
We prove several off-diagonal and pointwise estimates for singular integral operators that extend compactly on $L^{p}(\mathbb R^{n})$. 
\end{abstract}

\maketitle

\section{Introduction}

An operator
is said to satisfy an
off-diagonal estimate from $L^{p}(\mathbb R^{n})$ into $L^{q}(\mathbb R^{n})$ for $p,q>0$ if there exists a function  $G:[0,\infty )\rightarrow [0,\infty )$ 
vanishing at infinity such that 
\begin{equation*}\label{firstoff}
\| T(f\chi_{E})\chi_{F}\|_{L^{q}(\mathbb R^{n})}\lesssim G(\dist (E,F))\| f\|_{L^{p}(\mathbb R^{n})}
\end{equation*}
for all Borel sets $E,F\subset \mathbb R^{n}$ and all $f\in L^{p}(\mathbb R^{n})$, with implicit constant depending on the operator $T$ and the exponents $p,q$. 
Some authors distinguish between properly off-diagonal estimates, when $E\cap F=\emptyset$, and the so-called on-diagonal estimates, when $E\cap F\neq \emptyset$. However, we will not follow such convention and instead we will always call them off-diagonal estimates.

In the specific case of singular integral operators, the study focuses on the exponents $1\leq p=q<\infty $. Very often, off-diagonal bounds are considered in one of the two following dual forms:
$$
\| T(\chi_{I})\chi_{(\lambda I)^{c}}\|_{L^{1}(\mathbb R^{n})}\lesssim G(\lambda\ell(I))|I|,
\hskip 20pt
\frac{1}{|I|}\int_{I} |T(\chi_{(\lambda I)^{c}})(x)|dx\lesssim G(\lambda \ell(I))
$$
(see \cite{AM} and \cite{HyNa}),
for any cube $I\subset \mathbb R^{n}$ and $\lambda>1$,
where $|I|$ denotes the volume of the cube and $\lambda I$ is a concentric dilation of $I$.

While their use in Analysis is very classical, 
the interest for this type of inequalities in modern Harmonic Analysis 
renewed in the nineties after the publication of new proofs of the T(1) Theorem that used the wavelet decomposition approach (see \cite{AT} and \cite{CoifMeybook} for example). These proofs were based on the development of estimates of the form
\begin{equation}\label{offbound}
|\langle T(\psi_{I}),\psi_{J}\rangle |\lesssim  \Big(\frac{|J|}{|I|}\Big)^{\frac{1}{2}+\frac{\delta}{n}}
\Big(1+\frac{\dist(I,J)}{|I|^{\frac{1}{n}}}\Big)^{-(n+\delta)}
\end{equation}
for all cubes $I,J\subset \mathbb R^{n}$ with $|J|\leq |I|$, 
under the appropriate hypothesis on the operator $T$, the parameter $\delta>0$ and the functions $\psi_{I}, \psi_{J}$ involved.

The importance of off-diagonal inequalities lays mainly on two facts. On the one side, 
they are a satisfactory replacement for pointwise estimates of the operator kernel when these are not available or even when the operator kernel is unknown.
On the other side,  
they completely enclose the almost orthogonality properties of the operator.
For these reasons, they played a crucial role in the solution of the famous Kato's conjecture \cite{AHLMT} about boundedness of square root of elliptic operators and are nowadays extensively used in 
the study of second order elliptic operators. In the field, these estimates are typically established for one-parameter collections of operators $(T_{t})_{t>0}$ and the function $G$ also depends on the parameter $t$ in an appropriate manner
(see \cite{AKMP}, \cite{AKM2}, \cite{AKM}, \cite{HoMA}, \cite{HoMayMcI} and \cite{Ro}).
Finally, it is also worth mentioning that off-diagonal bounds 
provide very valuable information 
for the development of efficient algorithms to compress and rapidly evaluate 
discrete singular operators
(see \cite{BCR} and \cite{YQX}).

In the project {\em A characterization of compactness for singular integrals} I developed, 
a new $T(1)$ Theorem to characterize not only boundedness but also compactness of Singular Integral operators
$$
T(f)(x)=\int f(t)K(x,t)dt
$$
with $K$ a standard Calder\'on-Zygmund kernel. 
The main theorem provides sufficient and necessary conditions for compactness of Calder\'on-Zygmund operators in terms of the kernel decay and the action of the operator over special families of functions:

One of the goals of the current paper is to establish similar type of estimates for singular integral operators that can be extended compactly on $L^{p}(\mathbb R^{n})$ with $1<p<\infty $. 
In \cite{V}, the author proved a characterization of these operators based on a new type of off-diagonal estimates for Calder\'on-Zygmund operators. Now, in the current paper, we aim to improve these bounds in several ways and also obtain some new estimates. 
More explicitly, 
we show in section \ref{mainoff} that, in a broad sense and under the right hypotheses, 
these operators satisfy similar inequalites 
to (\ref{offbound}) but with a new factor $F$
that encodes the extra decay obtained as a consequence of their compactness properties:
\begin{equation}\label{compoffbound}
|\langle T(\psi_{I}),\psi_{J}\rangle |\lesssim  \Big(\frac{|J|}{|I|}\Big)^{\frac{1}{2}+\frac{\delta}{n}}
\Big(1+\frac{\dist(I,J)}{|I|^{\frac{1}{n}}}\Big)^{-(n+\delta)}
F(I,J).
\end{equation}
The focus of this work is actually placed on obtaining a sharp and as detailed as possible 
description of the function $F$ in the different cases under study. 

Furthermore, in section \ref{point&off} we establish pointwise estimates of the action of the operator 
over compactly supported functions. This allows to claim, in a broad sense as well, that the image of a bump
function adapted and supported in a cube behaves as a bump function adapted, although not supported, to the same cube. As before, these estimates explicitly state an extra decay
not present in the classical bounds that is again due to the compactness of the operator.

\section{Notation and Definitions}\label{ecandrdist}

We say that the set $I=\prod_{i=1}^{n}[a_{i},b_{i}]$ is a cube in $\mathbb R^{n}$
if the quantity $|b_{i}-a_{i}|$ is constant when varying the index $i$.
We denote by ${\mathcal Q}_{n}$ the family of all cubes in $\mathbb R^{n}$. 
For every cube $I\subset \mathbb R^{n}$, we denote its centre by $c(I)=((a_{i}+b_{i})/2)_{i=1}^{n}$, its side length by $\ell(I)=|b_{i}-a_{i}|$ and
its volume by $|I|=\ell(I)^{n}$. For any $\lambda >0$, we denote by $\lambda I$, the cube such that $c(\lambda I)=c(I)$ and $|\lambda I|=\lambda^{n}|I|$.

We write $|\cdot |_{p}$ for the $l^{p}$-norm in $\mathbb R^{n}$ with $1\leq p\leq \infty $ and $|\cdot |$ for the modulus of a complex number. Hopefully, the latter notation will not cause any confusion with the one used for the volume of a cube. We denote by $\mathbb B=[-1/2,1/2]^{n}$ and $\mathbb B_{\lambda }=\lambda \mathbb B=[-\lambda/2,\lambda /2]^{n}$.

Given two cubes $I,J\subset \mathbb R^{n}$,
we define $\langle I,J\rangle$ as the unique cube such that it contains $I\cup J$ with the smallest possible side length and
whose center has the smallest possible first coordinate. 
In the last section, this notation will be applied also to 
points, namely $\langle x,y\rangle $, as if they were considered to be degenerate cubes. 

We denote the side length of $\langle I,J\rangle$ by $\diam(I\cup J)$. 
Notice that
\begin{eqnarray*}
\diam(I\cup J) & \approx & \ell(I)/2+|c(I)-c(J)|_{\infty }+\ell(J)/2\\
& \approx & \ell(I)+\dist_{\infty}(I,J)+\ell(J)
\end{eqnarray*}
where $\dist_{\infty }(I,J)$ denotes the set distance between $I$ and $J$ calculated using the norm $| \cdot |_{\infty }$. Actually,
$$
\frac{1}{2}\diam(I\cup J)
\leq \frac{\ell(I)}{2}+|c(I)-c(J)|_{\infty }+\frac{\ell(J)}{2}
\leq \diam(I\cup J).
$$

We define
the relative distance between $I$ and $J$ by
$$
\rdist(I,J)=\frac{\diam(I\cup J)}{\max(\ell(I),\ell(J))}, 
$$
which is comparable to $\max(1,n)$ where $n$ is  
the smallest number of times the larger cube needs to be shifted a distance equal to its side length 
so that it contains the smaller one. The following equivalences hold:
\begin{eqnarray*}
\rdist(I,J) & \approx & 1+\max(\ell(I),\ell(J))^{-1}|c(I)-c(J)|_{\infty }\\
& \approx & 1+\max(\ell(I),\ell(J))^{-1}\dist_{\infty }(I,J).
\end{eqnarray*}

Finally, we define the eccentricity of $I$ and $J$ as
$$
\ec(I,J)=\frac{\min(|I|,|J|)}{\max (|I|,|J|)}.
$$

\begin{definition}
In order to characterize compactness of singular integral operators, we use two sets of 
auxiliary 
bounded functions $L,S, D: [0,\infty )\rightarrow [0,\infty )$ and 
$F: {\mathcal Q}_{n}\rightarrow [0,\infty )$
satisfying 
the following limits
\begin{equation}\label{limitD}
\lim_{x \rightarrow \infty }L(x)=\lim_{x \rightarrow 0}S(x)=\lim_{x \rightarrow \infty }D(x)=0, 
\end{equation}
\begin{equation}\label{limitsF}
\lim_{\ell(I)\rightarrow \infty }F(I)=\lim_{\ell(I)\rightarrow 0}F(I)=\lim_{|c(I)|_{\infty }\rightarrow \infty }F(I)=0.
\end{equation}
\end{definition}

\begin{remark}\label{constants}
Since any dilation ${\mathcal D}_{\lambda}L_{\lambda}(x)=L(\lambda^{-1}x)$, ${\mathcal D}_{\lambda}F(I)=F(\lambda^{-1}I)$ with $L$ and $F$ satisfying (\ref{limitD}), (\ref{limitsF}) respectively still satisfies the same limits, we will often omit all universal constants appearing in the arguments.
\end{remark}

\begin{definition}\label{compCZ}
Let $K:(\mathbb R^{n}\times \mathbb R^{n}) \setminus 
\{ (t,x)\in \mathbb R^{n}\times \mathbb R^{n} : t=x\} \to \mathbb C$. 

We say that $K$ is a
compact Calder\'on-Zygmund kernel if there exist constants $0<\delta \leq 1$, $C>0$ 
and functions $L$, $S$ and $D$ satisfying the limits in (\ref{limitD}), 
such that 
\begin{equation}\label{smoothcompact}
|K(t,x)-K(t',x')|
\leq
C\frac{(|t-t'|_{\infty }+|x-x'|_{\infty })^\delta}{|t-x|_{\infty }^{n+\delta}}
F_{K}(t,x)
\end{equation}
whenever $2(|t-t'|_{\infty }+|x-x'|_{\infty })<|t-x|_{\infty }$, with 
$$
F_{K}(t,x)=L(|t-x|_{\infty })S(|t-x|_{\infty })D(|t+x|_{\infty }).
$$

We say that $K$ is 
is a
standard Calder\'on-Zygmund kernel if (\ref{smoothcompact}) is satisfied with $F_{K}\equiv 1$.
\end{definition}

We first note that, without loss of generality, $L$ and $D$ can be assumed to be non-creasing
while $S$ can be assumed to be non-decreasing. This is possible because, otherwise, we can always define 
$$
L_{1}(x)=\sup_{y\in [x,\infty )}L(y)
\hskip30pt
S_{1}(x)=\sup_{y\in [0,x]}S(y)
\hskip30pt
D_{1}(x)=\sup_{y\in [x,\infty )}D(y)
$$
which bound above $L$, $S$ and $D$ respectively, satisfy the limits in (\ref{limitD}) and are non-creasing
or non-decreasing as requested.

On the other hand, we also denote
$$
F_{K}(t,t',x,x')=L_{2}(|t-x|_{\infty })S_{2}(|t-t'|_{\infty }+|x-x'|_{\infty })
D_{2}\Big(1+\frac{|t+x|_{\infty }}{1+|t-x|_{\infty }}\Big)
$$
and assume, in a similar way as before, that $L_{2}$ and $D_{2}$ are non-creasing while $S_{2}$ is non-decreasing. 
Then, as explained in \cite{V}, (\ref{smoothcompact}) is equivalent to the following smoothness condition 
\begin{equation}\label{smoothcompact2}
|K(t,x)-K(t',x')|
\leq
C\frac{(|t-t'|_{\infty }+|x-x'|_{\infty })^{\delta'}}{|t-x|_{\infty }^{n+\delta'}}F_{K}(t,t',x,x')
\end{equation}
whenever $2(|t-t'|_{\infty }+|x-x'|_{\infty })<|x-t|_{\infty }$, with possibly smaller $\delta'<\delta $. 
This resulting parameter $\delta'$ necessarily satisfies $\delta'<1$ since otherwise the kernel $K$ would be a constant function. 

The proof of the equivalence between both formulations appears in \cite{V}. However, 
to increase readability of the current paper, 
we sketch the proof that (\ref{smoothcompact}) implies (\ref{smoothcompact2}). 
For any $0<\epsilon <\delta $, let $\delta' = \delta - \epsilon$. Then, from (\ref{smoothcompact}) we can write
$$
|K(t,x)-K(t',x')|
\le C 
\frac{(|t-t'|_{\infty }+|x-x'|_{\infty })^{\delta'}}{|t-x|_{\infty }^{n+\delta' }}\tilde{F}_{K}(t,t',x,x')
$$
with 
$$
\tilde{F}_{K}(t,t',x,x')=\frac{(|t-t'|_{\infty }+|x-x'|_{\infty })^{\epsilon}}{|t-x|_{\infty }^{\epsilon}}
F_{K}(t,x).
$$
Then, the functions 
\begin{align*}
L_{2}(y) &= \sup_{|x-t|_{\infty } \geq y} \tilde{F}_{K}(t,t',x,x')^{1/3} \\
S_{2}(y) &= \sup_{|x-x'|_{\infty }+|t-t'|_{\infty } \leq y} \tilde{F}_{K}(t,t',x,x')^{1/3} \\ 
D_{2}(y) &= \sup_{1+\frac{|x+t|_{\infty }}{1+|x-t|_{\infty }} \geq y}\tilde{F}_{K}(t,t',x,x')^{1/3}
\end{align*}
satisfy all the required limits in (\ref{limitD}) and 
$$
\tilde{F}_{K}(t,t',x,x')\leq L_{2}(|t-x|_{\infty })S_{2}(|t-t'|_{\infty }+|x-x'|_{\infty })
D_{2}\Big(1+\frac{|t+x|_{\infty }}{1+|t-x|_{\infty }}\Big) .
$$

As also proved in \cite{V}, 
the smoothness condition 
(\ref{smoothcompact}) and the hypothesis ${\displaystyle \lim_{|t-x|_{\infty }\rightarrow \infty}K(t,x)=0}$ imply the classical decay condition
\begin{equation}\label{decaycompact}
|K(t,x)|
\lesssim
\frac{1}{|t-x|_{\infty }^{n}}
F_{K}(t,x)
\lesssim 
\frac{1}{|t-x|_{\infty }^{n}}
\end{equation}
for all $t,x\in \mathbb R^{n}$ such that $t\neq x$. 
Moreover, it is easy to see that we also get the decay
\begin{equation}\label{decaycompact2}
|K(t,x)|
\lesssim
\frac{1}{|t-x|_{\infty }^{n}}
L(|t-x|_{\infty })S(|t-x|_{\infty })D\Big(1+\frac{|t+x|_{\infty }}{1+|t-x|_{\infty }}\Big),
\end{equation}
which we will use later. Notice the change in the argument of $S$, which is now equal to the argument of $L$.

Finally, we define two more sets of auxiliary functions which we will use in the next section. 
First, 
\begin{equation}\label{F_K}
F_{K}(I_{1},I_{2},I_{3})=L(\ell(I_{1}))S(\ell(I_{2}))D(\rdist(I_{3},\mathbb B))
\end{equation}
and $F_{K}(I)=F_{K}(I,I,I)$. 
Second, 
\begin{equation}\label{F_K2}
\tilde{F}_{K}(I_{1},I_{2},I_{3})=L(\ell(I_{1}))S(\ell(I_{2}))\tilde{D}(\rdist(I_{3},\mathbb B))
\end{equation}
and $\tilde{F}_{K}(I)=\tilde{F}_{K}(I,I,I)$, where
\begin{equation}\label{tildeD}
\tilde{D}(\rdist (I,\mathbb B))=\sum_{j\geq 0}2^{-j\delta }D(\rdist (2^{j}I,\mathbb B)).
\end{equation}
We note that for fixed $\ell(I)$, the Lebesgue Dominated Theorem guarantees that $\tilde{D}$ satisfies 
${\displaystyle \lim_{|c(I)|_{\infty }\rightarrow \infty }\tilde{D}(\rdist (I,\mathbb B))=0}$.

\begin{definition}\label{intrep}
Let $T:\mathcal C_{0}(\R^{n})\to \mathcal C_{0}(\R^{n})'$ be a continuous linear operator. 
We say that
$T$ is associated with a Calder\'on-Zygmund kernel if there exists a function $K$ 
fulfilling Definition \ref{compCZ} such that 
the dual pairing satisfies the following
integral representation
$$
\langle T(f), g\rangle =\int_{\R^{n}}\int_{\R^{n}} f(t)g(x) K(x,t)\, dt \, dx
$$
for all functions $f,g\in {\mathcal C}_{0}(\mathbb R^{n})$
with disjoint compact supports.
\end{definition}

Clearly, the integral converges absolutely since, by (\ref{decaycompact}), we have for 
$d=\dist(\supp f,\supp g)>0$, 
$$
\Big| \int_{\mathbb R^{n}} \int_{\mathbb R^{n}} f(t)g(x)K(x,t)\, dtdx \Big|
\lesssim \| f\|_{L^{1}(\mathbb R^{n})}
\| g\|_{L^{1}(\mathbb R^{n})}
\frac{1}{d^{n}}.
$$

\begin{definition}
Let $0<p\leq \infty $. We say that a bounded function $\phi $ is
an $L^p(\mathbb R^{n})$-normalized bump function adapted to $I$ with constant $C>0$, decay $N\in \mathbb N$ and order $0$, if for all $x\in \mathbb R^{n}$
\begin{equation}\label{waveletdecay}
|\phi(x)|\le C|I|^{-\frac{1}{p}}\Big(1+\frac{|r-c(I)|_{\infty }}{\ell(I)}\Big)^{-N}.
\end{equation}

We say that a continuous bounded function $\phi $ is an $L^p(\mathbb R^{n})$-normalized bump function adapted to $I$ with constant $C>0$, decay $N\in \mathbb N$, order $1$ and parameter $0<\alpha\leq 1$, if (\ref{waveletdecay}) holds 
and for all $t,x\in \mathbb R^{n}$
\begin{equation}\label{waveletderivative}
|\phi(t)-\phi(x)|\le C\Big(\frac{|t-x|_{\infty }}{\ell(I)}\Big)^{\alpha }|I|^{-\frac{1}{p}}\sup_{r\in \langle t,x\rangle}
\Big(1+\frac{|r-c(I)|_{\infty }}{\ell(I)}\Big)^{-N},
\end{equation}
where $\langle t,x \rangle$ denotes the cube containing the points $t$ and $x$ with the smallest possible side length and whose centre has the smallest possible first coordinate. 
\end{definition}
Unless otherwise stated, we will assume the bump functions to be $L^2(\mathbb R^{n})$-normalized.

\begin{definition}\label{RC}
We say that a linear operator $T : \mathcal C_{0}(\R^{n}) \to \mathcal C_{0}(\R^{n})'$
satisfies the weak compactness condition, if
there exists a bounded function $F_{W}$ satisfying (\ref{limitsF}) and 
such that
for any cube $I\subset \mathbb R^{n}$ and any bump functions $\phi_I$, $\varphi_{I}$
adapted to $I$ with constant $C>0$, decay $N$ and order $0$, we have
$$
|\langle T(\phi_I),\varphi_{I}\rangle |
\lesssim C
F_{W}(I)
$$
where the implicit constant only depends on the operator $T$.
\end{definition}

As explained in \cite{V}, this definition admits several other reformulations, 
but they all essentially imply that the dual pairing $\langle T(\phi_I),\varphi_{I}\rangle $ tends to zero when the cube involved is large, small or far away from the origin. 

\begin{definition}
We define $\CMO(\mathbb R^{n})$ as the closure in $\BMO(\mathbb R^{n})$ of the space of continuous functions vanishing at infinity. 
\end{definition}

The following theorem, which is the main result in \cite{V}, 
characterizes compactness of Calder\'on-Zygmund operators. This is the reason why we say that  
the new off-diagonal bounds appearing in the current paper apply to operators that can be extended compactly on $L^{p}(\mathbb R^{n})$.
\begin{theorem}\label{Mainresult}
Let
$T$ be a linear operator associated with a standard Calder\'on-Zygmund kernel. 

Then, $T$ extends to a compact operator on $L^p(\mathbb R)$ for all $1<p<\infty$ if and only if $T$ 
is associated with a compact Calder\'on-Zygmund kernel and it satisfies
the weak compactness condition
and the cancellation conditions
$T(1), T^{*}(1) \in \CMO(\mathbb R)$. 
\end{theorem}

\section{Off-diagonal estimates for bump functions}\label{mainoff}
In the proof of Theorem \ref{Mainresult}, some off-diagonal estimates 
were developed. Now, we improve these inequalities in several directions:   
by extending the result to $\mathbb R^{n}$, 
by weakening the smoothness requirements of the bumps, 
by shortening the proof and by obtaining a sharper bound for functions with compact support.  

This is the purpose of the three propositions of this section, which describe the action of a compact singular integral operator over bump functions with or without zero mean properties respectively. 
Later, in section \ref{point&off}, we will use these bounds to obtain several pointwise bounds and other off-estimates of a more general type.

We first set up some notation that appears in the statements of the three results. 
We consider $K$ to be a compact Calder\'on-Zygmund kernel with parameter $0<\delta < 1$ and $T$ to be a linear operator with associated kernel $K$ satisfying the weak compactness condition.
We denote by $I\smland J$ and $I\smlor J$ the smallest and the largest of two given cubes $I,J$ respectively. That is, $I\smland J=J$, $I\smlor J=I$
if $\ell(J)\leq \ell(I)$, while $I\smland J=I$, $I\smlor J=J$, otherwise.
We also remind the notation of 
$F_{K}$, $F_{W}$ and $\tilde{F}_{K}$ provided in the previous section.


\begin{proposition}\label{symmetricspecialcancellation}
If 
the special cancellation conditions $T(1)=T^{*}(1)=0$ hold then, 
for all bump functions $\psi_{I}$, $\psi_{J}$ 
adapted and supported on $I$, $J$ respectively,
with constant $C>0$, order one, parameter $\alpha >\delta $ and such that $\psi_{I\smland J}$ has mean zero, 
\begin{equation}\label{twobump21}
|\langle T(\psi_{I}),\psi_{J}\rangle |\lesssim  C^{2} \frac{\ec(I,J)^{\frac{1}{2}+\frac{\delta}{n}}}{
\rdist(I,J)^{n+\delta}}
F(I, J)
\end{equation}
where $F$ is such that:
\begin{enumerate}
\item[\it i)]  $F(I, J)=F_{K}(\langle I,J\rangle ,I\smland J, \langle I,J\rangle )$ when $\rdist (I,J)> 3$,
\item[\it ii)]  $F(I, J)=
\tilde{F}_{K}(I\smlor J ,I\smland J, I\smlor J)+
F_{W}(I\smland J)+F_{K}(I\smland J,I\smland J,I\smlor J)$, otherwise.
\end{enumerate}

\end{proposition}

\begin{proposition}\label{symmetricspecialcancellation2}
For all bump functions $\psi_{I}$, $\psi_{J}$ 
adapted and supported on $I$, $J$ respectively,
with constant $C>0$, order one, parameter $\alpha >\delta $ and such that $\psi_{I\smland J}$ has mean zero, we have
\begin{equation}\label{twobump22}
|\langle T(\psi_{I}),\psi_{J}\rangle |\lesssim  C^{2} \frac{\ec(I,J)^{\frac{1}{2}}}{
\rdist(I,J)^{n+\delta}}
F(I, J)
\end{equation}
with $F(I, J)=\tilde{F}_{K}(I\smland J)+F_{W}(I\smland J)+F_{K}(I\smland J ,I\smland J, I\smlor J)$ when $\rdist (I,J)\leq 3$.\\
On the other hand, when $\rdist (I,J)>3$, inequality {\rm (\ref{twobump21})} still holds with the same  
$F(I, J)=F_{K}(\langle I,J\rangle ,I\smland J, \langle I,J\rangle )$. 
\end{proposition}

\begin{proposition}\label{symmetricspecialcancellation3}
For all bump functions $\psi_{I}$, $\psi_{J}$ 
adapted and supported on $I$, $J$ respectively,
with constant $C>0$ and order zero, we have
\begin{equation}\label{twobump23}
|\langle T(\psi_{I}),\psi_{J}\rangle |\lesssim  C^{2} \frac{\ec(I,J)^{\frac{1}{2}}}{
\rdist(I,J)^{n}}
\big(1+\big|\log \ecc(I,J)\big|^{\theta }\big)F(I, J)
\end{equation}
where
\begin{enumerate}
\item[\it i)]  $F(I, J)=F_{K}(\langle I,J\rangle )$ and $\theta=0$ when $\rdist (I,J)> 3$,
\item[\it ii)]  $F(I, J)=
F_{W}(I\smland J)+F_{K}(I\smland J,I\smlor J, I\smlor J)$ and $\theta=1$ when 
$\rdist (I,J)\leq 3$.
\end{enumerate}

\end{proposition}

In all cases, the implicit constants depend on the operator $T$ and the parameters $\delta $ and $\alpha $ but they are universal otherwise. Needless to say that the actual value appearing in the condition 
$\rdist (I,J)> 3$ plays no special role and it could be easily changed by any other value strictly larger than one. 

As mentioned before, 
Proposition \ref{symmetricspecialcancellation} is an improvement 
of the analog result in \cite{V}. The result has been extended to non-smooth bump functions
of several dimensions. At the same time, the proof has been largely simplified by using the extra hypothesis that 
the bump functions are compactly supported. Moreover, with this hypothesis, the last factor on the right hand side of the inequality turned out to be strictly smaller than the one appearing in \cite{V}.
In fact, when the bump functions are not longer compactly supported, as it happens in \cite{V}, the inequality 
(\ref{twobump21}) holds with a larger factor depending on six different cubes rather than only three cubes. 
Nevertheless, in both cases, the factors enjoy essentially the same properties 
and so, each of the two estimates suffices to prove compactness of the operator.

We also note that in Proposition \ref{symmetricspecialcancellation2}, the 
hypotheses that $T(1), T^{*}(1)\in \BMO(\mathbb R^{n})$ or $T(1), T^{*}(1)\in \CMO(\mathbb R^{n})$ are not needed. Moreover, in Proposition
\ref{symmetricspecialcancellation3}, the assumption of 
$T$ satisfying the special cancellation conditions $T(1)=T^{*}(1)=0$ does not lead to any further improvement.

\begin{notation}\label{not4proof}
For the following three proofs, we provide some common notation.  
For every cube $I\subset \mathbb R^{n}$, we denote by $\Phi_{I}\in \mathcal S(\mathbb R^{n})$ an 
$L^\infty$-normalized function adapted to $I$ 
with arbitrary large order and decay such that $0\leq \Phi_{I}\leq 1$, 
$\Phi_{I}=1$ in $2I$ and $\Phi_{I}=0$ in $(4I)^{c}$. 
This implies that $\Phi_{I}(x)=1$ for all $|x-c(I)|_{\infty }\leq \ell(I)$ while 
$\Phi_{I}(x)=0$ for all $|x-c(I)|_{\infty }> 2\ell(I)$.

As customary, we define the translation and dilation operators by ${\mathcal T}_{a}f(x)=f(x-a)$ and 
${\mathcal D}_{\lambda }f(x)=f(\lambda^{-1}x)$  respectively with $x,a\in \mathbb R^{n}$ and $\lambda>0$.
We also define $w_{I} (x)=1+\ell(I)^{-1}|x-c(I)|_{\infty }$
and for any function $\psi =\psi_{1}\otimes \psi_{2}$ of tensor product type, we write 
$\Lambda (\psi)=\langle T(\psi_{1}),\psi_{2}\rangle $. 

Finally, by symmetry
we can assume that $\ell(J)\leq \ell(I)$ and so, $I\smland J=J$ while $I\smlor J=I$.
\end{notation}

\begin{proof}[\it Proof of Proposition \ref{symmetricspecialcancellation}]
Let $\psi (t,x)=\phi_{I}(t)\psi_{J}(x)$ which, by hypothesis, is supported and adapted to $I\times J$ with constant $C^{2}$, decay $N$, order $1$, parameter $\alpha >\delta $ and, most importantly, it has mean zero in the variable $x$.

{\bf a)} We first assume that $3\ell(I)<\diam (I\cup J)$ which  implies $(5I)\cap J=\emptyset $ and so, 
$\diam(I\cup J)=\ell(I)/2+|c(I)-c(J)|_{\infty }+\ell(J)/2\leq \ell(I)+|c(I)-c(J)|_{\infty }$.

Then, since $|t-c(I)|_{\infty }\leq \ell(I)/2$, we have 
\begin{equation}\label{equivdiam}
|t-c(J)|_{\infty }
\leq \ell(I)/2+|c(I)-c(J)|_{\infty }\leq \diam(I\cup J)
\end{equation}
and
$$
|t-c(J)|_{\infty }\geq |c(I)-c(J)|_{\infty }-|t-c(I)|_{\infty }
\geq \ell(I)+|c(I)-c(J)|_{\infty }-3\ell(I)/2
$$
$$
\geq \diam(I\cup J)-\diam(I\cup J)/2 =\diam(I\cup J)/2.
$$

On the other hand, 
$3\ell(I)<\diam(I\cup J)\leq \ell(I)+|c(I)-c(J)|_{\infty }$ also implies $2\ell(I)<|c(I)-c(J)|_{\infty }$ and 
since $|x-c(J)|_{\infty }\leq \ell(J)/2$, we get  
$$
|t-c(J)|_{\infty }\geq |c(I)-c(J)|_{\infty }-|t-c(I)|_{\infty }
$$
$$
\geq 2\ell(I)-\ell(I)/2
\geq 3\ell(J)/2\geq 3|x-c(J)|_{\infty }.
$$
The last inequality implies that the support of $\psi $ is disjoint with the diagonal and so, 
we can use the Calder\'on-Zygmund kernel representation to write
$$
\Lambda(\psi)=\int \int \psi(t,x) K(t, x)\, dtdx
=\int_{J}\int_{I} \psi(t,x) (K(t, x)-K(t, c(J)))\, dtdx
$$
where the second equality is due to the zero mean of $\psi$ in the variable $x$.
Now, we denote $Q_{I,J}=\{t\in \mathbb R^{n}: \diam(I\cup J)/2<|t-c(J)|_{\infty }\leq \diam(I\cup J)\}$. Then,  
by the smoothness condition (\ref{smoothcompact2}) of a compact Calder\'on-Zygmund kernel and 
the monotonicity properties of $L$, $S$ and $D$, we bound as follows:
\begin{align*}
&\hskip10pt|\Lambda (\psi)|
\lesssim 
\int_{J}
\int_{I\cap Q_{I,J}}\hspace{-.5cm}
|\psi (t, x)| \frac{|x-c(J)|_{\infty }^{\delta}}{|t-c(J)|_{\infty }^{n+\delta }}
\\
&\hskip60pt 
L(|t-c(J)|_{\infty })S(|x-c(J)|_{\infty })D\Big(1+\frac{|t+c(J)|_{\infty }}{1+|t-c(J)|_{\infty }}\Big)\, dx dt
\\
&\lesssim \| \psi\|_{L^{1}(\mathbb R^{2n})}\frac{\ell(J)^{\delta }}{\diam(I\cup J)^{n+\delta }}
L(\diam(I\cup J))S(\ell(J))
D\Big(1+\frac{|c(J)|_{\infty }}{1+\diam(I\cup J)}\Big)
\\
&\lesssim C^{2}|I|^{\frac{1}{2}}|J|^{\frac{1}{2}}
\frac{\ell(J)^{\delta }}{\diam(I\cup J)^{n+\delta }}
L(\ell(\langle I,J\rangle))S(\ell(J))
D(\rdist(\langle I,J\rangle ,\mathbb B))
\\
&= C^{2} \Big(\frac{|J|}{|I|}\Big)^{\frac{1}{2}+\frac{\delta }{n}}
\Big(\frac{\diam(I\cup J)}{\ell(I)}\Big)^{-(n+\delta )}
F_{K}(\langle I,J\rangle, J,\langle I,J\rangle ) 
\end{align*}
as stated. To completely finish this case, 
we explain in more detail the reasoning used to obtain the bounds for $D$ used in the second and third inequalities above.
Since
$
|x|_{\infty }\leq (|x-t|_{\infty }+|x+t|_{\infty })/2
$, 
we have
$$
1+\frac{|x|_{\infty }}{1+|t-x|_{\infty }}\leq 1+\frac{1}{2}+\frac{|t+x|_{\infty }}{1+|t-x|_{\infty }}
\leq \frac{3}{2}\Big(1+\frac{|t+x|_{\infty }}{1+|t-x|_{\infty }}\Big).
$$
Then, in the domain of integration,
$$
1+\frac{|c(J)|_{\infty }}{1+\diam(I\cup J)}
\leq 1+\frac{|c(J)|_{\infty }}{1+|t-c(J)|_{\infty }}
\leq \frac{3}{2}\Big(1+\frac{|t+c(J)|_{\infty }}{1+|t-c(J)|_{\infty }}\Big)
$$
and, since $D$ is non-creasing, we have
$$
D\Big(1+\frac{|t+c(J)|_{\infty }}{1+|t-c(J)|_{\infty }}\Big)\leq D\Big(1+\frac{|c(J)|_{\infty }}{1+\diam(I\cup J)}\Big)
$$
omitting constants. 

On the other hand,
since $|c(I)|_{\infty }-|c(J)|_{\infty }\leq |c(I)-c(J)|_{\infty }\leq \diam(I\cup J)$, we can bound below
the numerator of the argument of $D$ in the last expression by
\begin{align*}
1+\diam(I\cup J)+|c(J)|_{\infty } 
&\geq 1+\frac{1}{2}\diam(I\cup J)+\frac{1}{2}(|c(I)|_{\infty }-|c(J)|_{\infty })+|c(J)|_{\infty }
\\
&=1+\frac{1}{2}\diam(I\cup J)+\frac{1}{2}(|c(I)|_{\infty }+|c(J)|_{\infty })
\\
&\geq \frac{1}{2}\big(1+\diam(I\cup J)+\frac{1}{2}|c(I)+c(J)|_{\infty }\big).
\end{align*}
Then, 
\begin{align*}
1+\frac{|c(J)|_{\infty }}{1+\diam(I\cup J)})
&\geq \frac{1}{2}\frac{1+\diam(I\cup J)+|c(I)+c(J)|_{\infty }/2}{1+\diam(I\cup J)}
\\
&\geq \frac{1}{3}\Big(\frac{3}{2}+\frac{|c(I)+c(J)|_{\infty }/2}{\diam(I\cup J)}\Big).
\end{align*}
Finally, since $|(c(I)+c(J))/2-c(\langle I,J\rangle)|_{\infty }\leq \ell(\langle I,J\rangle)/2$
and $\ell(\langle I,J\rangle)=\diam(I\cup J)$, we bound below previous expression by
$$
\frac{1}{3}\Big(\frac{3}{2}+\frac{|c(\langle I,J\rangle )|_{\infty }}{\diam(I\cup J)}-\frac{1}{2}\Big)
\geq \frac{1}{3}\Big(1+\frac{|c(\langle I,J\rangle )|_{\infty }}{\max(\ell(\langle I,J\rangle),1)}\Big)
= \frac{1}{3}\rdist(\langle I,J\rangle ,\mathbb B ).
$$

{\bf b)} We now assume that $\diam (I\cup J)\leq 3\ell(I)$ which implies $1\leq \rdist(I,J)\leq 3$.
In this case, we first show that we can assume $\psi(c(J),x)=0$ for any $x\in \mathbb R^{n}$.
This assumption comes from the substitution of $\psi(t,x)$ by
\begin{equation}\label{subtraction1}
\psi(t,x)-({\mathcal T}_{c(J)}{\mathcal D}_{\ell(I)}\Phi)(t) \psi(c(J),x)
\end{equation}
where $\Phi =\Phi_{\mathbb B}$ as described in Notation \ref{not4proof}. 
Then, we only need to prove that the subtracted term satisfies the desired bound. 

We denote $\tilde{\psi}(x)=\psi(c(J),x)$. Since $\psi_{I}$ and $\psi_{J}$ are adapted to $I$ and $J$ respectively with constant $C>0$ and decay $N$ for any $N\in \mathbb N$, we have
$$
|\tilde{\psi}(x)|\leq C^{2}|I|^{-\frac{1}{2}}|J|^{-\frac{1}{2}}w_{J}(x)^{-N}.
$$
Then, $\| \tilde{\psi}\|_{L^{1}(\mathbb R^{n})}\leq C^{2}|I|^{-\frac{1}{2}}|J|^{\frac{1}{2}}$. 
We also recall that $\tilde{\psi}$ is supported on $J$ and has mean zero.

Now, we write $\lambda=\ell(I)/\ell(J)\geq 1$ and take $k\in \mathbb N$ so that $2^{k}\leq \lambda < 2^{k+1}$. Then,  
$$
{\cal T}_{c(J)}{\mathcal D}_{\ell(I)}\Phi
={\cal T}_{c(J)}{\mathcal D}_{\lambda \ell(J)}\Phi .
$$
To simplify notation, we write
$\Phi_{0}={\cal T}_{c(J)}{\mathcal D}_{\ell(I)}\Phi \in \mathcal S(\mathbb R^{n})$ and 
$\Phi_{1}=1-\Phi_{0}$. We note that $\Phi_{1}$ is a smooth bounded function supported on $|t-c(J)|_{\infty }>\lambda \ell(J)$.  
By the classical theory, we know that $T(1)$ can be defined as a distribution acting on the space of compactly supported functions with mean zero in the following way
\begin{align*}
\langle T(1),\tilde{\psi}\rangle 
&=\langle T(\Phi_{0}),\tilde{\psi}\rangle 
+\int \! \int \Phi_{1} (t)\tilde{\psi}(x)
K(t,x)dt dx
\\
&=\langle T(\Phi_{0}),\tilde{\psi}\rangle 
+\int \! \int \Phi_{1} (t)\tilde{\psi}(x)
(K(t,x)-K(t, c(J)))dt dx
\end{align*}
where the second equality is due to the mean zero of $\tilde{\psi }$.
Notice that, since
$
|x-c(J)|_{\infty }\leq \ell(J)/2
\leq \ell(J)2^{k-1}\leq 2^{-1}|t-c(J)|_{\infty }
$, 
the supports of $\Phi_1$ and $\tilde{\psi}$
are disjoint and so the integral in the first line converges absolutely. Then, the hypothesis that $T(1)=0$ implies 
$$
\langle T(\Phi_0),\tilde{\psi}\rangle 
=-\int \! \int \Phi_{1} (t) \tilde{\psi}(x)
(K(t,x)-K(t, c(J)))dt dx.
$$
Moreover, since $2|x-c(J)|_{\infty }<|t-c(J)|_{\infty }$, we can use the 
the smoothness condition (\ref{smoothcompact2}) of a compact Calder\'on-Zygmund kernel
to write
\begin{align*}
|\langle T(\Phi_0), \tilde{\psi}\rangle\rangle |&\leq \int_{J} \int_{\lambda \ell(J)<|t-c(J)|_{\infty }}
 |\Phi_{1}(t)||\tilde{\psi}(x)| \frac{|x-c(J)|_{\infty }^{\delta }}{|t-c(J)|_{\infty }^{n+\delta}}
\\
&\hskip40pt L(|t-c(J)|_{\infty })S(|x-c(J)|_{\infty })
D\Big(1+\frac{|t-c(J)|_{\infty }}{1+|t-c(J)|_{\infty }}\Big)
\, dtdx.
\end{align*}
Then, by the reasoning applied in the previous case, we have
\begin{align*}
|\langle T(\Phi_0), \tilde{\psi}\rangle\rangle |
&\lesssim \| \tilde{\psi}\|_{L^1(\mathbb R^{n})}\ell(J)^{\delta}L(\lambda \ell(J))S(\ell(J))
\\
&\int\limits_{2^{k}\ell(J)<|t-c(J)|_{\infty }} 
\frac{1}{|t-c(J)|_{\infty }^{n+\delta}}D\Big(1+\frac{|c(J)|_{\infty }}{1+|t-c(J)|_{\infty }}\Big)dt .
\end{align*}
Now, we rewrite the last integral as
\begin{align*}
\sum_{j\geq k}\hskip5pt
&\int\limits_{2^{j}\ell(J)<|t-c(J)|\leq 2^{j+1}\ell(J)} 
\frac{1}{|t-c(J)|_{\infty }^{n+\delta}}D\Big(1+\frac{|c(J)|_{\infty }}{1+|t-c(J)|_{\infty }}\Big)
dt
\\
&\lesssim \sum_{j\geq k}
D\Big(1+\frac{|c(J)|_{\infty }}{1+2^{j+1}\ell(J)}\Big)\frac{(2^{j+1}\ell(J))^{n}}{(2^{j}\ell(J))^{n+\delta}}
\\
&\lesssim \ell(J)^{-\delta }\sum_{j\geq k}2^{-j\delta }
D\Big(1+\frac{|c(J)|_{\infty }}{1+2^{j+1}\ell(J)}\Big)
\\
&= \ell(J)^{-\delta }
2^{-k\delta }\sum_{j\geq 0}2^{-j\delta }D\Big(1+\frac{|c(J)|_{\infty }}{1+2^{j+k+1}\ell(J)}\Big)
\\
&\lesssim \ell(I)^{-\delta }
\sum_{j\geq 0}2^{-j\delta }D\Big(1+\frac{|c(J)|_{\infty }}{1+2^{j}\ell(I)}\Big)
\\
&\lesssim \ell(I)^{-\delta }
\sum_{j\geq 0}2^{-j\delta }D(\rdist (2^{j}I,\mathbb B))
\\
&
=\ell(I)^{-\delta }
\tilde{D}(\rdist (I,\mathbb B))
\end{align*}
with $\tilde{D}$ as in (\ref{tildeD}). We now detail the step taken in the last inequality:
since $|c(I)-c(J)|_{\infty }\leq \diam(I\cup J)\leq 3\ell(I)$, we have that 
$$
4\Big(1+\frac{|c(J)|_{\infty }}{1+2^{j}\ell(I)}\Big)\geq 4+\frac{|c(I)|_{\infty }}{1+2^{j}\ell(I)}-\frac{3}{2^{j}}
\geq 1+\frac{|c(I)|_{\infty }}{1+2^{j}\ell(I)}=D(\rdist (2^{j}I,\mathbb B)) .
$$
Then, we write
\begin{align*}
|\langle T(\Phi_0), \tilde{\psi}\rangle\rangle |
&\lesssim C^{2}|I|^{-\frac{1}{2}}|J|^{\frac{1}{2}}
\left(\frac{\ell(J)}{\ell(I)}\right)^{\delta }
L(\lambda \ell(J))S(\ell(J))
\tilde{D}(\rdist (I,\mathbb B))
\\
&= C^{2}  \left(\frac{|J|}{|I|}\right)^{\frac{1}{2}+\frac{\delta }{n}}
\tilde{F}_{K}(I,J,I)
\end{align*}
which is the first term in the stated bound. This finishes the justification of the assumption
$\psi(c(J),x)=0$ for any $x\in \mathbb R^{n}$.

Now, we decompose $\psi$ in the following way:
\begin{align*}
\psi&=\psi_{out}+\psi_{in}
\\
\psi_{in}(t,x)&= \psi(t,x)\Phi_{3J}(t).
\end{align*}

b1) We first prove that $\psi_{in}$ is adapted to $J\times J$ with order zero, decay $N$ and constant
$
C^{2}\left(|J|/|I|\right)^{\frac{1}{2}+\frac{\delta}{n}}
$.

By the assumption $\psi(c(J),x)=0$, the fact that
$\psi$ is supported and adapted to $I\times J$ with order one and parameter $\alpha $ and that 
$\psi_{in}$ is supported on $3J\times J$,
we have for all  $t\in 3J$ and all $x\in J$,
\begin{align*}
|\psi_{in}&(t,x)|
=|\psi(t,x)-\psi(c(J),x)|\Phi_{3J}(t)
\\
&\lesssim C\Big(\frac{|t-c(J)|_{\infty }}{\ell(I)}\Big)^{\alpha }  
|I|^{-\frac{1}{2}} 
\hspace{-.2cm} 
\sup_{r\in \langle t, c(J) \rangle}\hspace{-.1cm} \Big(1+\frac{|r-c(I)|_{\infty }}{\ell(I)}\Big)^{-N}\hspace{-.3cm} \chi_{3J}(t)
C|J|^{-\frac{1}{2}}w_{J} (x)^{-N}
\\
&\lesssim C^{2}\Big(\frac{\ell(J)}{\ell(I)}\Big)^{\alpha } |I|^{-\frac{1}{2}} 
\chi_{3J}(t)
|J|^{-\frac{1}{2}}w_{J} (x)^{-N}
\\
&\lesssim C^{2 }\left(\frac{|J|}{|I|}\right)^{\frac{1}{2}+\frac{\delta }{n}}
|J|^{-\frac{1}{2}} w_{J} (t)^{-N}
|J|^{-\frac{	1}{2}}w_{J} (x)^{-N}
\end{align*}
since $\delta <\alpha $ and $|J|\leq |I|$. 
Notice that we also used $|t-c(J)|_{\infty }\leq 3\ell(J)/2$. 

Therefore $\psi_{in} $ is adapted to $J\times J$ with order zero and the stated constant and so, 
by the weak compactness property of $T$ we get
$$
|\Lambda (\psi_{in})| 
\lesssim C^{2} \left(\frac{|J|}{|I|}\right)^{\frac{1}{2}+\frac{\delta }{n}}
F_{W}(J)
$$
which ends this case.

b2) We now work with $\psi_{out}$. 
In this case, by the extra assumption again and the support of $\psi $, we have the following decay
\begin{align}\label{uno0}
\nonumber
|\psi_{out}(t,x)|&\leq |\psi (t,x)-\psi(c(J),x)|
\\
\nonumber
&
\lesssim C \Big(\frac{|t-c(J)|_{\infty }}{\ell(I)}\Big)^{\alpha }|I|^{-1/2}  
\chi_{I}(t)C|J|^{-1/2}w_{J} (x)^{-N}
\\
&= C^{2}\Big(\frac{|t-c(J)|_{\infty }}{\ell(I)}\Big)^{\alpha }
\varphi_{I\times J}(t,x)
\end{align}
by denoting 
$\varphi_{I\times J}(t, x)=|I|^{-1/2}\chi_{I}(t)|J|^{-1/2}w_{J} (x)^{-N}$. 

Due to the support of $\psi$, we get
$|t-c(I)|_{\infty }\leq \ell(I)/2$ while, by the calculations in \eqref{equivdiam} and the hypothesis of this case, we also have 
$$
|t-c(J)|_{\infty }
\leq \diam(I\cup J)\leq 3\ell(I).
$$
Moreover, due to the support of $\psi_{out}$, we have $|t-c(J)|_{\infty }\geq 3\ell(J)$ and
$|x-c(J)|_{\infty }\leq \ell(J)/2$.
The last two inequalities imply $2|x-c(J)_{\infty }|<|t-c(J)|_{\infty }$
and so,
we use the integral representation 
and the mean zero of $\psi_{out}$ in the variable $x$
to write
$$
\Lambda (\psi_{out})
=\int \int \psi_{out}(t,x) (K(t,x)-K(t,c(J))) \, dtdx .
$$

This, together with the bound of $\psi_{out}$ calculated in (\ref{uno0})
and the smoothness condition (\ref{smoothcompact2}) of a compact Calder\'on-Zygmund kernel,
allow us to bound in the following way:
\begin{align*}
|\Lambda (\psi_{out})|
&\lesssim \frac{C^{2}}{\ell(I)^{\alpha }}
\int\limits_{J} \int\limits_{3\ell(J)<|t-c(J)|_{\infty }\leq 3\ell(I)}\hspace{-1cm} |t-c(J)|_{\infty }^{\alpha }
\varphi_{I\times J}(t, x)
|K(t, x)-K(t, c(J))|
\, dtdx
\\
&\lesssim \frac{C^{2}}{\ell(I)^{\alpha }}
\| \varphi_{I\times J}\|_{L^{\infty }(\mathbb R^{2n})}
\int\limits_{J} \int\limits_{3\ell(J)<|t-c(J)|_{\infty }\leq 3\ell(I)}\hspace{-.6cm} |t-c(J)|_{\infty }^{\alpha }
\frac{|x-c(J)|_{\infty }^\delta}{|t-c(J)|_{\infty }^{n+\delta}}
\\
&\hskip30pt
L(|t-c(J)|_{\infty })
S(|x-c(J)|_{\infty })D\Big(1+\frac{|t+c(J)|_{\infty }}{1+|t-c(J)|_{\infty }}\Big) 
\, dtdx
\\
&\lesssim \frac{C^{2}}{\ell(I)^{\alpha }}|I|^{-1/2}|J|^{-1/2}L(\ell (J))S(\ell (J))
D\Big(1+\frac{|c(J)|_{\infty }}{1+\ell(I)}\Big)
\\
&\int\limits_{|x-c(J)|<\ell(J)/2} |x-c(J)|_{\infty }^\delta dx 
\int\limits_{3\ell(J)<|t-c(J)|_{\infty }\leq 3\ell(I)}\frac{1}{|t-c(J)|_{\infty }^{n+\delta-\alpha }}dt.
\end{align*}

The first integral can be bounded by
$$
\int\limits_{|x-c(J)|<\ell(J)/2} |x-c(J)|_{\infty }^\delta dx 
\lesssim |J|^{1+\frac{\delta }{n}}
$$
and, since $\delta <\alpha $,
the second integral is bounded by 
$$
\int\limits_{3\ell(J)<|t-c(J)|_{\infty }\leq 3\ell(I)}\frac{1}{|t-c(J)|_{\infty }^{n+\delta-\alpha }}dt
\lesssim (3\ell(I))^{\alpha -\delta}-(3\ell(J))^{\alpha -\delta}
\lesssim \ell(I)^{\alpha -\delta} .
$$ 
On the other hand, since $|c(I)-c(J)|_{\infty }\leq \diam(I\cup J)\leq 3\ell(I)$, we have as before 
$$
4(1+\frac{|c(J)|_{\infty }}{1+\ell(I)})
\geq 4+\frac{|c(I)|_{\infty }}{1+\ell(I)}-3
\gtrsim \rdist (I,\mathbb B) .
$$

Finally then,
\begin{align*}
|\Lambda (\psi_{out})|
&\lesssim C^{2}|I|^{-1/2}|J|^{-1/2}
L(\ell (J))S(\ell (J))
D(\rdist (I,\mathbb B))
|J|^{1+\frac{\delta }{n}}
\ell(I)^{-\delta}
\\
&\leq C^{2}\left( \frac{|J|}{|I|} \right)^{\frac{1}{2} + \frac{\delta }{n}} 
F_{K}(J,J,I).
\end{align*}

\end{proof}

\begin{proof}[\it Proof of Proposition \ref{symmetricspecialcancellation2}]
As before, $\psi (t,x)=\phi_{I}(t)\psi_{J}(x)$ is supported and adapted to $I\times J$ with constant $C^{2}$, 
decay $N$, order $1$, parameter $\alpha >\delta$ 
 and it has mean zero in the variable $x$. We divide the proof into the same cases as before. 


{\bf a)} When $3\ell(I)<\diam (I\cup J)$, 
exactly the same reasoning of case a) in the proof of Proposition \ref{symmetricspecialcancellation2} holds since 
the only properties needed are the mean zero of $\psi_{J}$ and the smoothness property of the compact 
Calder\'on-Zygmund kernel.

{\bf b)} We now assume that $\diam (I\cup J)\leq 3\ell(I)$.
As before, we first show that we can assume $\psi(c(J),x)=0$ for any $x\in \mathbb R^{n}$.
This assumption comes again from the substitution of $\psi(t,x)$ by
\begin{equation}\label{subtraction2}
\psi(t,x)-({\mathcal T}_{c(J)}{\mathcal D}_{\ell(I)}\Phi)(t) \psi(c(J),x) .
\end{equation}
But now we need to prove that the subtracted term satisfies the desired bound without the use of the condition $T(1)=0$. We remind the notation $\Phi_{0}={\mathcal T}_{c(J)}{\mathcal D}_{\ell(I)}\Phi$ with 
$\Phi=\Phi_{\mathbb B}$ as defined in Notation \ref{not4proof}.

We denote $\tilde{\psi}(x)=\psi(c(J),x)$ which, as before, 
satisfies the decay
$$
|\tilde{\psi}(x)|\leq C^{2}|I|^{-1/2}|J|^{-1/2}w_{J}(x)^{-N}
$$
and so, it is a 
bump function supported and adapted to $J$ with order zero and constant $C^{2}|I|^{-1/2}$.

Let $J_{k}=2^{k}J$ for $k\in \mathbb N$, $k\geq 0$ and 
let $\Phi_{J_{k}}$ be bump functions $L^{\infty}$-adapted to $J_{k}$ and supported on $4J_{k}$
as defined in Notation \ref{not4proof}.  
We define now $\psi_{0}=\Phi_{J_{0}}$ and $\psi_{k}=\Phi_{J_{k}}-\Phi_{J_{k-1}}$ for $k\geq 1$
which satisfy $\sum_{k\geq 0}\psi_{k}(x)=1$ for all $x\in \mathbb R^{n}$. Therefore, we have
$$
|\langle T(\Phi_{0}),\tilde{\psi}\rangle |
\leq \sum_{k\geq 0}|\langle T(\Phi_{0}\cdot \psi_{k}),\tilde{\psi}\rangle |
$$
with a finite sum due to the compact support of $\Phi_{0}$ which implies $2^{k-1}\ell(J)
\leq |t-c(J)|_{\infty }\leq 2\ell(I)$. 

Now, for $k=0$, since $\Phi_{0} \cdot |J|^{-1/2}\Phi_{J_{0}}$ is supported on $4J$
and $L^{2}$-adapted to $J$, we can apply weak compactness condition to obtain
$$
|\langle T(\Phi_{0}\cdot \Phi_{J_{0}}),\tilde{\psi}\rangle |
\lesssim C^{2}|I|^{-1/2}|J|^{1/2}F_{W}(J) .
$$

When $k\geq1$, due to the supports of $\psi_{k}$ and $\tilde{\psi}$, 
we have that $2^{k-1}\ell(J)<|t-c(J)|_{\infty }<2^{k+1}\ell(J)$ and
$|x-c(J)|_{\infty }<\ell(J)/2$ respectively. This implies that 
$2|x-c(J)|_{\infty }\leq \ell(J)\leq |t-c(J)|_{\infty }$ and so, 
we can use the integral representation, the mean zero of $\tilde{\psi}$
and 
the smoothness condition (\ref{smoothcompact2}) of a compact Calder\'on-Zygmund kernel
to write 
\begin{align*}
 \sum_{k\geq 1}|\langle T(\Phi_{0}\psi_{k}),\tilde{\psi}\rangle |
&=\sum_{k\geq 1}\Big| \int \! \int
\Phi_{0}(t)\psi_{k}(t)\tilde{\psi}(x)
(K(t,x)-K(t,c(J))dtdx\Big|
\\
&\leq \sum_{k\geq 1}\int\limits_{J}
\int\limits_{2^{k-1}\ell(J)<|t-c(J)|_{\infty }<2^{k+1}\ell(J)}
|\tilde{\psi}(x)| \frac{|x-c(J)|_{\infty }^{\delta}}{|t-c(J)|_{\infty }^{n+\delta }}
\\
&\hskip20pt L(|t-c(J)|_{\infty })S(|x-c(J)|_{\infty })
D(1+\frac{|t+c(J)|_{\infty }}{1+|t-c(J)|_{\infty }})dtdx
\\
&
\leq \| \tilde{\psi}\|_{L^{1}(\mathbb R^{n})}\ell(J)^{\delta }
L(\ell(J))S(\ell(J))
\\
&\sum_{k\geq 1}D\Big(1+\frac{|c(J)|_{\infty }}{1+2^{k+1}\ell(J)}\Big)
\int\limits_{2^{k-1}\ell(J)<|t-c(J)|_{\infty }}
\frac{1}{|t-c(J)|_{\infty }^{n+\delta }}dt
\\
&\lesssim C^{2}\left(\frac{|J|}{|I|}\right)^{\frac{1}{2}}\ell(J)^{\delta }L(\ell(J))S(\ell(J))
\sum_{k\geq 1}\frac{D(\rdist(2^{k}J,\mathbb B))}{2^{k\delta }\ell(J)^{\delta}}
\\
&\lesssim C^{2}\left(\frac{|J|}{|I|}\right)^{\frac{1}{2}}
L(\ell(J))S(\ell(J))\tilde{D}(\rdist(J,\mathbb B))
\\
&
= C^{2}\left(\frac{|J|}{|I|}\right)^{\frac{1}{2}}
\tilde{F}_{K}(J,J,J)
\end{align*}
which is the first term of the stated bound.

This finishes the justification of the assumption
$\psi(c(J),x)=0$. From here, the proof that $\Lambda (\psi_{in })$ satisfies the required bounds 
follows exactly the same steps 
as the one in cases b1) and b2) in the proof of Proposition \ref{symmetricspecialcancellation}.

\end{proof}

\begin{proof} [\it Proof of Proposition \ref{symmetricspecialcancellation3}]
Now, the function $\psi (t,x)=\phi_{I}(t)\psi_{J}(x)$ is supported and adapted to $I\times J$ with constant $C^{2}$, decay $N$ and order zero but it does not necessarily have mean zero. 

{\bf a)} As before, we first assume that $3\ell(I)<\diam (I\cup J)$. By the calculations in the proof of Proposition
\ref{symmetricspecialcancellation}, we have 
$$
\diam(I\cup J)/2\leq |t-c(J)|_{\infty }\leq \diam(I\cup J)
$$
and
$
|x-c(J)|_{\infty }\leq \ell(J)/2
$, 
which again imply
$
|t-c(J)|_{\infty }\geq 3|x-c(J)|_{\infty }
$. 
Then, the support of $\psi $ is disjoint with the diagonal
and we can use the Calder\'on-Zygmund kernel representation to write
$$
\Lambda(\psi)=\int \! \int \psi(t,x) K(t,x)\, dtdx.
$$

Now, with the same notation $Q_{I,J}=\{ t\in \mathbb R^{n}: \diam(I\cup J)/2<|t-c(J)|_{\infty }\leq \diam(I\cup J)\}$ 
and the kernel decay described in (\ref{decaycompact2}), 
we bound as follows:
\begin{align*}
|\Lambda (\psi)|
&\leq \int_{J}
\int_{I\cap Q_{I,J}}
|\psi (t,x)||K(t,x)|\, dtdx
\\
&
\lesssim \int_{J}
\int_{I\cap Q_{I,J}}
|\psi (t, x)|\frac{1}{|t-c(J)|_{\infty }^{n}}
\\
&\hskip50pt 
L(|t-c(J)|_{\infty })S(|t-c(J)|_{\infty })
D\Big(1+\frac{|t+c(J)|_{\infty }}{1+|t-c(J)|_{\infty }}\Big)
dx dt
\\
&\lesssim \frac{\| \psi\|_{L^{1}(\mathbb R^{2n})}}{\diam(I\cup J)^{n}}L(\diam(I\cup J))S(\diam(I\cup J))
D\Big(1+\frac{|c(J)|_{\infty }}{1+\diam(I\cup J)}\Big)
\\
&\lesssim 
\frac{C^{2}|I|^{\frac{1}{2}}|J|^{\frac{1}{2}}}{\diam(I\cup J)^{n}}
L(\ell(\langle I\cup J\rangle ))S(\ell(\langle I\cup J\rangle ))
D(\rdist (\langle I\cup J\rangle ,\mathbb B))
\\
&= C^{2}\Big(\frac{|J|}{|I|}\Big)^{\frac{1}{2}}
\Big(\frac{\diam(I\cup J)}{\ell(I)}\Big)^{-n}
F_{K}(\langle I, J\rangle ,\langle I, J\rangle ,\langle I, J\rangle ) .
\end{align*}

{\bf b)} We now assume that $\diam (I\cup J)\leq 3\ell(I)$ 
and we decompose $\psi$ in the same way as before:
$
\psi=\psi_{out}+\psi_{in}
$
with
$
\psi_{in}(t,x)= \psi(t,x)\Phi_{3J}(t)
$ 
and divide the analysis into the same cases.

%

b1) We claim that $\psi_{in}$ is adapted to $J\times J$ with order zero and constant
$
C^{2}\left(|J|/|I|\right)^{\frac{1}{2}}
$.
Since $\psi $ is adapted to $I\times J$ and 
$\psi_{in}$ is supported on $3J\times J$,
we have for all  $t\in 3J$ and all $x\in J$,
$$
|\psi_{in}(t,x)|
=|\psi(t,x)|\Phi_{3J}(t)
\lesssim C^{2}|I|^{-\frac{1}{2}}\chi_{J}(t) |J|^{-\frac{1}{2}}w_{J}(x)^{-N}
$$
$$
\lesssim C^{2} \left(\frac{|J|}{|I|}\right)^{\frac{1}{2}}
|J|^{-\frac{1}{2}} w_{J} (t)^{-N}
|J|^{-\frac{	1}{2}}w_{J} (x)^{-N}.
$$
This proves the claim and so, 
by the weak compactness property of $T$, we get 
$$
|\Lambda(\psi_{in})| 
\lesssim C^{2} \left(\frac{|J|}{|I|}\right)^{\frac{1}{2}}
F_{W}(J).
$$

b2) We now work with $\psi_{out}$
for which we have the decay
$$
|\psi_{out}(t,x)|\leq |\psi (t,x)|\lesssim |I|^{-1/2}  
\chi_{I}(t)|J|^{-1/2}w_{J}(x)^{-N}.
$$

Moreover, the calculations in the analog case b2) of the proof of Proposition \ref{symmetricspecialcancellation} show that 
on the support of $\psi$ we have that $|t-c(I)|_{\infty }\leq \ell(I)/2$ and
$$
|t-c(J)|_{\infty }
\leq \diam(I\cup J)\leq 3\ell(I)
$$
while due to the support of $\psi_{out}$, we also have 
$|t-c(J)|_{\infty }\geq 3\ell(J)$ and
$|x-c(J)|_{\infty }\leq \ell(J)/2$.
Then, $2|x-c(J)_{\infty }|<|t-c(J)|_{\infty }$
and 
we can use the integral representation to bound in the following way:
\begin{align*}
|\Lambda (\psi_{out})|&=\Big|\int \int \psi_{out}(t,x) K(t,x) \, dtdx \Big|
\\
&
\leq
\int\limits_{J}\hskip2pt \int\limits_{3\ell(J)<|t-c(J)|_{\infty }\leq 3\ell(I)}
|\psi_{out}(t,x)|
|K(t,x)-K(t,c(J))|
\, dtdx
\\
&\hskip10pt +
\int\limits_{J}\hskip2pt \int\limits_{3\ell(J)<|t-c(J)|_{\infty }\leq 3\ell(I)}
|\psi_{out}(t,x)|
|K(t,c(J))|
\, dtdx .
\end{align*}

The first term can be bounded by a constant times
$$
C^{2}\left( \frac{|J|}{|I|} \right)^{\frac{1}{2}+\frac{\delta}{n}}L(\ell(J))S(\ell(J))D(\rdist(I,\mathbb B))
$$
exactly in the same way as in case b2) in the proof of Proposition \ref{symmetricspecialcancellation}. 
On the other hand, from 
the decay of a compact Calder\'on-Zygmund kernel stated in (\ref{decaycompact2}) 
and the decay for $\psi_{out}$, we can bound 
the second term by a constant times
\begin{align*}
\int\limits_{J} &\int\limits_{3\ell(J)<|t-c(J)|_{\infty }\leq 3\ell(I)}
|\psi_{out}(t, x)|
\frac{1}{|t-c(J)|_{\infty }^{n}}
\\
&\hskip50pt L(|t-c(J)|_{\infty })S(|t-c(J)|_{\infty })
D\Big(1+\frac{|t+c(J)|_{\infty }}{1+|t-c(J)|_{\infty }}\Big)
\, dtdx
\\
&\leq \| \psi_{out}\|_{L^{\infty }(\mathbb R^{2n})}
L(\ell(J))S(\ell(I))D\Big(1+\frac{|c(J)|_{\infty }}{1+\ell(I)}\Big)
\\
&\hskip50pt|J| \int\limits_{3\ell(J)<|t-c(J)|_{\infty }\leq 3\ell(I)}\hspace{-.2cm}
\frac{1}{|t-c(J)|_{\infty }^{n}}
\, dt
\\
&\leq C^{2}|I|^{-\frac{1}{2}}|J|^{-\frac{1}{2}}
L(\ell(J))S(\ell(I))D(\rdist(I,\mathbb B)) 
|J| (\log(3\ell(I))-\log(3\ell(J)))
\\
&= C^{2} \left( \frac{|J|}{|I|} \right)^{\frac{1}{2}}\log \left(\frac{\ell(I)}{\ell(J)}\right)
\, F_{K}(J,I,I) .
\end{align*}

\end{proof}

\section{Pointwise and off-diagonal estimates for general functions}\label{point&off}
In this last section, 
we provide several pointwise estimates of the action of the operator over general functions and over bump functions, Proposition \ref{pointwiseexpression} and Corollary \ref{corpointwiseexpression} respectively. 
Moreover, in Proposition \ref{lastoff} we prove a new off-diagonal inequality for general functions.

We start with some technical results which, despite being  
well-known for bounded singular operators,
we hereby reproduce here their proofs for compact singular operators in order to highlight the role played by compactness in the gain of decay and smoothness. 

\begin{lemma}\label{pointdef0}
Let $T$ be a linear operator associated with a standard Calder\'on-Zygmund kernel. 
Let $\Phi \in \mathcal S(\mathbb R^{n})$ such that it is positive, it is supported on 
$\mathbb B=[-\frac{1}{2},\frac{1}{2}]^{n}$ and $\int \Phi(x)dx=1$. We denote $\Phi_{x,\epsilon }(y)=\epsilon^{-n}\Phi(\epsilon^{-1}(y-x))$.

Let $f$ be an integrable function with compact support in a cube $I\subset \mathbb R^{n}$. 
Then, for all $x\notin 3I$ there exists the limit of $\langle T(f), \Phi_{x,\epsilon }\rangle $ when $\epsilon $ tends to zero.
\end{lemma}
\begin{definition}\label{pointdef} By previous lemma, we can define
$$
T(f)(x)=\lim_{\epsilon \rightarrow 0}\langle T(f), \Phi_{x,\epsilon }\rangle .
$$
\end{definition}
\begin{proof}
We check that  
$(\langle T(f), \Phi_{x,\epsilon }\rangle)_{\epsilon>0}$ is a Cauchy sequence.
Let $x\in \mathbb R^{n}\backslash (3I)$ fixed and we choose $\epsilon_{1},\epsilon_{2}<2\ell(I)/5$. 

Then, for all $t\in \supp f$ we have
$\ell(I)<|t-x|_{\infty }$ while for all we get $y\in \supp \Phi_{x,\epsilon_{i}}$, 
$|y-x|_{\infty }\leq \epsilon_{i} /2<\ell(I)/2$. Both inequalities imply 
$$
|t-y|_{\infty }\geq  |t-x|_{\infty }- |x-y|_{\infty }\geq \ell(I)/2>0.
$$ 
Hence, 
$f(t)$ and 
$\phi_{x,\epsilon_{i}}(y)$ have disjoint compact supports and, by the integral representation, we can write
$$
\langle T(f), \Phi_{x,\epsilon_{i} }\rangle
=\int \! \int f(t)\Phi_{x,\epsilon_{i} }(y)K(t,y)dtdy
$$
and so,
$$
\langle T(f), \Phi_{x,\epsilon_{1} }\rangle -
\langle T(f), \Phi_{x,\epsilon_{2} }\rangle 
=\int \! \int f(t)\Phi(y)(K(t,x+\epsilon_{1}y)-K(t,x+\epsilon _{2}y))dtdy .
$$

Now, for all
$y\in \supp \Phi$, we have $|y|_{\infty }\leq 1/2$ and so, 
\begin{align*}
2|(\epsilon_{1}-\epsilon_{2}) y|_{\infty }&<\epsilon_{1}+\epsilon_{2}<\frac{4}{5}\ell(I)
\leq \ell(I)-\frac{\epsilon_{1}}{2}
< |t-x|_{\infty }-\frac{\epsilon_{1}}{2}
\\
&\leq |t-x-\epsilon_{1}y|_{\infty }+\epsilon_{1}|y|_{\infty }-\frac{\epsilon_{1}}{2}
\leq |t-x-\epsilon_{1}y|_{\infty }.
\end{align*}

Therefore, 
we can apply the smoothness condition of the kernel to bound in the following way:
$$
|\langle T(f), \Phi_{x,\epsilon_{1} }\rangle -
\langle T(f), \Phi_{x,\epsilon_{2} }\rangle |
\leq \int \! \int |f(t)| |\Phi(y)|
\frac{|\epsilon_{1}-\epsilon_{2}|^{\delta }|y|_{\infty }^{\delta }}{|t-x-\epsilon_{1}y|_{\infty }^{n+\delta}} dt dy
$$
$$
\leq |\epsilon_{1}-\epsilon_{2}|^{\delta }
\| \Phi\|_{L^{\infty }(\mathbb R^{n})}\int |f(t)|\int\limits_{\ell(I)/2<|t-\epsilon_{1}y-x|_{\infty }} 
\frac{1}{|t-x-\epsilon_{1}y|_{\infty }^{n+\delta}} dt dx
$$
$$
\lesssim |\epsilon_{1}-\epsilon_{2}|^{\delta }\| f\|_{L^{1}(\mathbb R^{n})}\frac{1}{\ell(I)^{\delta}}.
$$
Notice we used that
$
|t-x-\epsilon_{i}y|_{\infty }\geq  |t-x|_{\infty }- \epsilon_{i}|y|_{\infty }\geq \ell(I)- \epsilon_{i}/2\geq \ell(I)/2
$.  
This proves that $(\langle T(f), \Phi_{x,\epsilon }\rangle)_{\epsilon>0}$ is Cauchy.
\end{proof}

In all forthcoming results, we consider $T$ to be a linear operator associated with a compact Calder\'on-Zygmund kernel $K$ with parameter $0<\delta <1$. We do not assume on $T$ any other hypotheses like 
weak boundedness or weak compactness or $T(1)$ belonging to any space in particular.

\begin{lemma}\label{contTf} 
Let $f$ an integrable function with compact support in a cube $I$. 
Then, 
$T(f)$ admits the following representation as a function
\begin{equation}\label{intrep}
T(f)(x)=\int f(t)K(x,t)dt
\end{equation}
for all $x\notin 3I$. Moreover, $T(f)$ is H\"older-continuous in $\mathbb R^{n}\backslash (3I)$ satisfying 
\begin{align*}
|T(f)(x)-T(f)(x')|
&\lesssim L(\ell(\langle I,x\rangle ))D(\rdist(\langle I,x\rangle ,\mathbb B))
\\
&\hskip30pt \frac{|x-x'|_{\infty }^{\delta }}{\ell(I)^{\delta}}
S(|x-x'|_{\infty })
|I|^{-1}\|f\|_{L^{1}(\mathbb R^{n})}.
\end{align*}
\end{lemma}
\begin{remark}
Notice that if $S(x)\leq |x|_{\infty }^{\beta}$ with $\beta>0$, then, $T(f)$ is 
H\"older-continuous with parameter $\delta +\beta $ which is better than in the case when $T$ is only a bounded singular integral operator.
\end{remark}
\begin{proof}
We first check the integral representation. We note that the integral in the right hand side of (\ref{intrep}) converges absolutely since by hypothesis $|t-x|_{\infty }\geq \dist_{\infty }(x,I)\geq \ell (I)$ and so, 
$$
\Big| \int f(t)K(x,t)dt\Big|
\leq \|f\|_{1}\frac{1}{|I|}.
$$
Now, let $x\in \mathbb R^{n}\backslash (3I)$, $\epsilon<2\ell(I)/5$ be fixed and
$\Phi_{x,\epsilon }$ as in Lemma \ref{pointdef0}. As before, for $t\in \supp f$, 
$y\in \supp \Phi_{x,\epsilon}$, we have
$\ell(I)<|t-x|_{\infty }$, $|y-x|_{\infty }\leq \epsilon /2<\ell(I)/2$ and
$
|t-y|_{\infty }\geq  |t-x|_{\infty }- |x-y|_{\infty }\geq \ell(I)/2>0
$. 
Hence, 
$f(t)$ and 
$\phi_{x,\epsilon }(y)$ have disjoint compact supports and we can write
$$
\langle T(f), \Phi_{x,\epsilon }\rangle
=\int \! \int f(t)\Phi_{x,\epsilon }(y)K(t,y)dtdy
= \int \! \int f(t)\Phi(y)K(t,x+\epsilon y)dtdy .
$$
Moreover, for  all $y\in \supp \Phi$ we have
$2|\epsilon y|_{\infty }<\epsilon <\ell(I)\leq |t-x|_{\infty }$ and so, 
by the smoothness condition of the kernel, we bound as follows:
\begin{align*}
\Big|\langle T(f), \Phi_{x,\epsilon }\rangle &-\int f(t)K(t,x)dt\Big|
\\
&=\Big| \int \! \int f(t)\Phi(y)K(t,x+\epsilon y)dtdy-\int \! \int f(t)\Phi (y)K(t,x)dtdy\Big|
\\
&
\leq \int \! \int |f(t)| |\Phi(y)| |K(t,x+\epsilon y)-K(t,x)|dtdy
\\
&\lesssim \int \! \int |f(t)| \frac{\epsilon |y|_{\infty }^{\delta }}{|t-x|_{\infty }^{n+\delta }}F_{K}(t,x)dtdy
\\
&
\lesssim \frac{\epsilon }{|I|^{1+\frac{\delta }{n}}}\int |f(t)|dt 
\end{align*}
which tends to zero as required.

Now, we check the H\"older-continuity of $T(f)$ in $\mathbb R^{n}\backslash 3I$.  
For all $x,x' \in \mathbb R^{n}\backslash (3I)$
with $|x-x'|_{\infty }<\ell(I)/2$ we have  
$2|x-x'|_{\infty }<\ell(I)\leq |x-t|_{\infty }$ and so, 
\begin{align*}
|T(f)(x)&-T(f)(x')|=\Big| \int f(t)(K(t,x)-K(t,x'))dt\Big| 
\\
&\lesssim \int |f(t)|\frac{|x-x'|_{\infty }^{\delta }}{|t-x|_{\infty }^{n+\delta }}
L(|t-x|_{\infty })S(|x-x'|_{\infty })
D\Big(1+\frac{|t+x|_{\infty }}{1+|t-x|_{\infty }}\Big)
dt
\\
&\lesssim |I|^{-1}\|f\|_{L^{1}(\mathbb R^{n})}
 \frac{|x-x'|_{\infty }^{\delta }}{\ell(I)^{\delta}}
 L(\ell(I))S(|x-x'|_{\infty })
D(\rdist(I,\mathbb B)) .
\end{align*}
Notice that, since $|t-c(I)|_{\infty }\leq \ell(I)/2$, $|t-x|_{\infty }\leq \ell(I)/2+|c(I)-x|_{\infty }\leq \ell(\langle I,x\rangle )$, 
$|t+x|_{\infty }\geq |x+c(I)|_{\infty }-\ell(I)/2$ and $|c(\langle I,x\rangle )-(x+c(I))/2|\leq \ell(\langle I,x\rangle )/2$, 
we have 
\begin{align*}
1+\frac{|t+x|_{\infty }}{1+|t-x|_{\infty }}&\gtrsim \frac{1}{4}\Big(4+\frac{2|c(\langle I,x\rangle )|_{\infty }-\frac{3}{2}(\ell(\langle I,x\rangle )+1)}{1+\ell(\langle I,x\rangle )}\Big)
\\
&\gtrsim \frac{5}{2}+2\frac{|c(\langle I,x\rangle )|_{\infty }}{1+\ell(\langle I,x\rangle )}
\gtrsim 2\rdist(\langle I,x\rangle ,\mathbb B) .
\end{align*}

\end{proof}

The following proposition will allow us to describe, in Corollary \ref{corpointwiseexpression}, the behavior of 
$T(\phi_{I})$ for any given bump function $\phi_{I}$ adapted and supported on a cube $I$,
stating explicitly a gain in the decay and smoothness that depends on the decay and smoothness of the operator kernel.

\begin{proposition}\label{pointwiseexpression}
Let $f$ be an integrable function supported on a cube $I$. Then, 
$$
|T(f)(x)|\lesssim w_{I}(x)^{-n}
F_{K}(\langle I, x\rangle)
|I|^{-1}\|f\|_{L^{1}(I)}
$$ 
for all $x\in \mathbb R^{n}$ such that $x\notin 5I$. Moreover, 
\begin{multline*}
|T(f)(x)-T(f)(x')|
\lesssim \frac{|x-x'|_{\infty }^{\delta }}{\ell(I)^{\delta }}w_{I}\Big(\frac{x+x'}{2}\Big)^{-(n+\delta )}\\
L(\ell(\langle I, x\rangle))S(|x-x'|_{\infty })
D(\rdist(\langle I, x\rangle,\mathbb B))
|I|^{-1}\|f\|_{L^{1}(I)}
\end{multline*}
for all $x,x'\in \mathbb R^{n}$ such that $|x-x'|_{\infty }<\ell(I)/2$ and 
$\langle x,x'\rangle\cap 5I=\emptyset $.
\end{proposition}
\begin{proof} 
We consider $\epsilon >0$ small enough so that it satisfies  
several inequalities stated along the proof. 
We denote by $\mathbb B_{\epsilon }=[-\epsilon/2,\epsilon/2]^{n}$ and, 
given $x\in \mathbb R^{n}$, we define $J=x+\mathbb B_{\epsilon }$. Let also  
$\varphi_{J}=C|J|^{-1}w_{J}(x)^{-N}$ be a positive bump function $L^{1}(\mathbb R^{n})$-adapted to $J$ with order one, decay $N$ and constant $C$ such that $\int \varphi_{J}(x)dx=1$. 

First, $\epsilon $ can be taken so that $\ell(J)\leq \ell(I)$. Moreover, 
the hypothesis $x\notin 5I$ implies that $\diam(I\cup J)>3\ell(I)$. Then, 
from the proof of Proposition \ref{symmetricspecialcancellation3} in its case a), we have 
\begin{eqnarray*}
|\langle T(f),\varphi_{J}\rangle |
&\lesssim &\| f\|_{L^{1}(\mathbb R^{n})}
\| \varphi_{J}\|_{L^{1}(\mathbb R^{n})}\diam(I\cup J)^{-n}F_{K}(\langle I, J\rangle)\\
&\lesssim & \rdist(I,J)^{-n}F_{K}(\langle I, J\rangle)|I|^{-1}\| f\|_{L^{1}(\mathbb R^{n})} .
\end{eqnarray*}

On the other hand, for $\epsilon $ small enough, we have 
$$
\ell (\langle I, x\rangle)\leq \ell (\langle I,J\rangle)\leq \ell (\langle I,x\rangle)+\epsilon/2
\leq 2\ell (\langle I,x\rangle)
$$ 
and 
$$
\frac{1}{2}\rdist(\langle I,x\rangle,\mathbb B)\lesssim \rdist(\langle I,J\rangle,\mathbb B)
\leq \rdist(\langle I,x\rangle,\mathbb B).
$$ 
These inequalities imply that
$$
F_{K}(\langle I,J\rangle)=L(\ell(\langle I,J\rangle))S(\ell(\langle I,J\rangle))D(\rdist(\langle I,J\rangle,\mathbb B))
\leq F_{K}(\langle I,x \rangle )
$$
omitting constants.

Now, taking limit when $\epsilon$ tends to zero, we get  
$$
\lim_{\epsilon \rightarrow 0}\langle T(f),\varphi_{J}\rangle
=\lim_{\epsilon \rightarrow 0}\int \int f(t) \varphi_{J}(y)K(t,y)dtdy
=\int f(t)K(t,x)dt=T(f)(x)
$$
by Lemma \ref{contTf} .
Finally then, 
\begin{eqnarray*}
|T(f)(x)|
&\lesssim &\rdist(I, x)^{-n}F_{K}(\langle I, x\rangle)|I|^{-1}\| f\|_{L^{1}(\mathbb R^{n})}\\
&\lesssim &\Big(1+\frac{|x-c(I)|_{\infty }}{\ell(I)}\Big)^{-n}F_{K}(\langle I, x\rangle)|I|^{-1}\| f\|_{L^{1}(\mathbb R^{n})}
\end{eqnarray*}
which proves the first inequality. 

We prove now the second one.
Let $x,x'\in \mathbb R^{n}$ as stated and let  $c=(x+x')/2$. Again, we consider $\epsilon >0$ to be small enough for our purposes. 
We define $J_{1}=x+\mathbb B_{\epsilon }$,
$J_{2}=x'+\mathbb B_{\epsilon }$ and the functions $\varphi_{J_{i}}$ for $i=1,2$ as before. Then, 
the function 
$\varphi_{J_{1}}-\varphi_{J_{2}}$ is supported on
$\langle J_{1}, J_{2}\rangle$ 
and it has mean zero. 

The hypotheses on $x$, $x'$ 
imply $|x-x'|_{\infty }<\ell(I)/2$ and $\diam(I\cup \langle x,x'\rangle)>3\ell(I)$. Then, for all $t\in I$ 
we have 
$|t-c|_{\infty }>\diam(I\cup \langle x,x'\rangle)/2>\ell(I)$.
Moreover, we can take 
$0<\epsilon <|x-x'|_{\infty }/2$ small enough so that 
$\diam(I\cup \langle J_{1},J_{2}\rangle )>3\ell(I)$, 
$|t-c|_{\infty }\geq \diam (I\cup \langle J_{1},J_{2}\rangle)/2$
and 
$J_{1}\cap J_{2}=\emptyset $. The latter property further implies that $\ell(\langle J_{1},J_{2}\rangle)
=|x-x'|_{\infty }+\epsilon < \ell(I)$.

On the other hand, since $c=c(\langle J_{1}, J_{2}\rangle)$, we have 
$|y-c|_{\infty }\leq \ell(\langle J_{1}, J_{2}\rangle)/2$ 
for all $y\in \langle J_{1}, J_{2}\rangle$.
Also notice that $\ell(J_{1})=\ell(J_{2})$. 

Finally then, 
$2|x-c|_{\infty }=|x-x'|_{\infty }<\ell(I)\leq |t-c|_{\infty }$, which allows to use the integral representation and the smoothness property of the compact Calder\'on-Zygmund kernel.

After establishing all these inequalities, we can 
repeat the proof of case a) in Proposition \ref{symmetricspecialcancellation} to obtain
\begin{align*}
|\langle T(f),&\varphi_{J_{1}}-\varphi_{J_{2}}\rangle |
=\Big|\int f(t) (\varphi_{J_{1}}(y)-\varphi_{J_{2}}(y))
(K(t,y)-K(t,c))dtdy \Big|
\\
&\leq \int |f(t)| |\varphi_{J_{1}}(y)-\varphi_{J_{2}}(y)|
\frac{|y-c|_{\infty }^{\delta }}{|t-c|_{\infty }^{n+\delta }}
\\
&\hskip20pt L(|t-c|_{\infty })S(|y-c|_{\infty })
D\Big(1+\frac{|t+c|_{\infty }}{1+|t-c|_{\infty }}\Big)
dtdy
\\
&\lesssim \| f\|_{L^{1}(\mathbb R^{n})} \| \varphi_{J_{1}}\|_{L^{1}(\mathbb R^{n})}
\frac{\ell(\langle J_{1}, J_{2}\rangle)^{\delta }}{\diam(I\cup \langle J_{1}, J_{2}\rangle)^{n+\delta }}
\\
&\hskip20pt L(\ell(\langle I, J_{1}\cup J_{2}\rangle))
S(\ell(\langle J_{1},J_{2}\rangle))
D(\rdist (\langle I, J_{1}\cup J_{2}\rangle),\mathbb B))
\\
&\lesssim 
\Big(\frac{\ell(\langle J_{1},J_{2}\rangle)}{\ell(I)}\Big)^{\delta }
\rdist(I,\langle J_{1}, J_{2}\rangle)^{-(n+\delta )}
\\
&\hskip20pt L(\ell(\langle I, J_{1}\cup J_{2}\rangle))
S(\ell(\langle J_{1},J_{2}\rangle))
D(\rdist (\langle I, J_{1}\cup J_{2}\rangle),\mathbb B))
|I|^{-1}\| f\|_{L^{1}(\mathbb R^{n})} .
\end{align*}

Now, we have the following relationships:
\begin{align*}
\ell (\langle I,x\rangle)&\leq \ell (\langle I,J_{1}\cup J_{2}\rangle)\leq \ell (\langle I,x\rangle)+|x-x'|_{\infty }+\epsilon/2\\
&\leq \ell (\langle I,x\rangle)+\ell (I)+\ell(I)/4\leq 3\ell (\langle I, x\rangle),\\
\ell(\langle J_{1},J_{2}\rangle)&\leq 2|x-x'|_{\infty },\\
\frac{1}{2}\rdist(\langle I,\{ x\}\rangle,\mathbb B)&\lesssim \rdist(\langle I,J_{1}\cup J_{2}\rangle,\mathbb B)\leq \rdist(\langle I,\{ x\}\rangle,\mathbb B),\\
\rdist(I,\langle J_{1}, J_{2}\rangle )&
\geq \rdist(I,c )\gtrsim 1+\frac{|\frac{x+x'}{2}-c(I)|_{\infty }}{\ell(I)}.
\end{align*}

These inequalities imply
$$
L(\ell(\langle I, J_{1}\cup J_{2}\rangle))
S(\ell(\langle J_{1},J_{2}\rangle))
D(\rdist (\langle I, J_{1}\cup J_{2}\rangle),\mathbb B))
$$
$$
\leq 
L(\ell(\langle I, x\rangle))
S(|x-x'|_{\infty })
D(\rdist (\langle I, x\rangle),\mathbb B))
$$
omitting constants. 

Finally then, we take limit when $\epsilon$ tends to zero to get
\begin{multline*}
|T(f)(x)-T(f)(x')|
\lesssim \frac{|x-x'|_{\infty }^{\delta }}{\ell(I)^{\delta }}
\Big(1+\frac{|\frac{x+x'}{2}-c(I)|_{\infty }}{\ell(I)}\Big)^{-(n+\delta )}\\
L(\ell(\langle I, x\rangle))
S(|x-x'|_{\infty })
D(\rdist (\langle I, x\rangle,\mathbb B))
|I|^{-1}\| f\|_{L^{1}(\mathbb R^{n})}
\end{multline*}
which proves the second inequality. 
\end{proof}

\begin{corollary}\label{corpointwiseexpression}
Let $\phi_{I}$ be a bump function adapted and supported to $I$ with 
constant $C>0$, decay $N$ and order zero. Then, in $\mathbb R^{n}\backslash 5I$,
$T(\phi_{I})$ satisfies the definition of a bump function adapted to $I$ with 
constant $C>0$, decay $n$, order one and parameter $\delta $, plus an extra factor in decay due to compactness. 
\end{corollary}
\begin{proof} We simply rewrite the statement of Proposition \ref{pointwiseexpression}:
$$
|T(\phi_{I})(x)|\lesssim C|I|^{-\frac{1}{2}}w_{I}(x)^{-n}F_{K}(\langle I, x\rangle)
$$
and
\begin{multline*}
|T(\phi_{I})(x)-T(\phi_{I})(x')|
\lesssim C\frac{|x-x'|_{\infty }^{\delta }}{\ell(I)^{\delta }}|I|^{-\frac{1}{2}}w_{I}\Big(\frac{x+x'}{2}\Big)^{-(n+\delta )}\\
L(\ell(\langle I, x\rangle))S(|x-x'|_{\infty })
D(\rdist(\langle I, x\rangle,\mathbb B))
\end{multline*}
for all $x,x'\in \mathbb R^{n}$ such that $x\notin 5I$, $|x-x'|_{\infty }<\ell(I)/2$ and $\langle x,x'\rangle \cap 5I=\emptyset $. 
\end{proof}

Notice that, even though $\phi_{I}$ has compact support and 
decays at infinity as fast as $|x|^{-N}$ for any large $N>0$, 
$T(\phi_{I})$ does not have in general compact support and its decay is only comparable to $|x|^{-n}$. Both facts are typical of bounded singular integral operators. However, if the operator is associated with a compact 
Calder\'on-Zgymund kernel, the decay of $T(\phi_{I})$ improves depending on the rate of decay of the factor $L(\ell(\langle I, x\rangle))$ when $x$ tends to infinity. 

On the other hand, note 
the gain in smoothness with respect bounded singular integrals
provided by the factor $|x-x'|_{\infty }^{\delta }S(|x-x'|_{\infty })$.


We show now an off-diagonal estimate for general functions 
deduced directly from the previous pointwise bound.

\begin{proposition}\label{lastoff}
Let $f$ an integrable function supported on a cube $I$. Then, for all $1<p<\infty $ and all $\lambda >1$, we have
$$
\| T(f)\chi_{(\lambda I)^{c}}\|_{L^{p}(\mathbb R^{n})}\lesssim \frac{1}{(1+\lambda )^{\frac{n}{p'}}}
\sup_{x\in (\lambda I)^{c}}F_{K}(\langle I, x\rangle) \| f\|_{L^{p}(\mathbb R^{n})}.
$$
\end{proposition}

\begin{proof} From Proposition \ref{pointwiseexpression}, we have
\begin{eqnarray*}
\| T(f)\chi_{(\lambda I)^{c}}\|_{L^{p}(\mathbb R^{n})}^{p}&\lesssim &\int_{(\lambda I)^{c}}
\frac{F_{K}(\langle I, x\rangle)^{p}}{(1+\ell(I)^{-1}|x-c(I)|_{\infty })^{np}}dx |I|^{-p}\| f\|_{L^{1}(I)}^{p}
\\
&\lesssim &
\frac{\ell(I)^{n}}{(1+\lambda )^{n(p-1)}}
\sup_{x\in (\lambda I)^{c}}F_{K}(\langle I, x\rangle)^{p}
 |I|^{-p}\| f\|_{L^{1}(I)}^{p}
\\
&=&\frac{1}{(1+\lambda )^{n(p-1)}}
\sup_{x\in (\lambda I)^{c}}F_{K}(\langle I, x\rangle)^{p}
 |I|^{-(p-1)}\| f\|_{L^{1}(I)}^{p}
\end{eqnarray*}
which, by H\"older, is smaller than the right hand side of the stated inequality. 
\end{proof}

We end the paper by adding few remarks to the previous proposition. 
We first note that, in the particular case of $f$ being a bump function, we obtain
$$
\| T(\phi_{I})\chi_{(\lambda I)^{c}}\|_{L^{p}(\mathbb R^{n})}
\lesssim |I|^{\frac{1}{p}-\frac{1}{2}}\frac{1}{(1+\lambda )^{\frac{n}{p'}}}
\sup_{x\in (\lambda I)^{c}}F_{K}(\langle I, x\rangle) .
$$

We also remind that,
for bounded but not compact singular integral operators, the analog of Proposition \ref{lastoff} implies that 
for a fixed cube $I$ and $\| f\|_{L^{p}(I)}\leq 1$, we have
$$
\lim_{\lambda \rightarrow \infty }\| T(f)\chi_{(\lambda I)^{c}}\|_{L^{p}(\mathbb R^{n})}=0
$$
with a rate of decay at most of order $\lambda^{\frac{n}{p'}}$. 
However, for compact singular integral operators, the extra factor stated in Proposition \ref{lastoff} ensures that 
there is always 
an extra gain in decay. 
To see this, 
we note that 
for all $x\in (\lambda I)^{c}$ we have $\lambda I\subset 3\langle I, x\rangle $ and so,
$\lambda \ell(I)\leq 3\ell(\langle I, x\rangle )$.
Hence, 
$
F_{K}(\langle I, x\rangle)\lesssim L(\ell(\langle I, x\rangle ))
\lesssim L(\lambda \ell(I))
$
and since $\lim_{\lambda \rightarrow \infty } L(\lambda \ell(I))=0$, the rate of decay is now 
at worst as fast as $\lambda^{\frac{n}{p'}}L(\lambda \ell(I))$.

Finally, since we also have the bound
$$
F_{K}(\langle I, x\rangle)\lesssim 
L(\ell(I))D\Big(1+\frac{|c(I)|_{\infty }}{1+\lambda \ell(I)}\Big)
$$
we deduce that for fixed $\lambda$ and $\| f\|_{L^{p}(I)}\leq 1$ we have that
$$
\displaystyle \lim_{\ell(I)\rightarrow \infty }\| T(f)\chi_{(\lambda I)^{c}}\|_{L^{p}(\mathbb R^{n})}
=0
$$
while, for fixed $\lambda$, $\| f\|_{L^{p}(I)}\leq 1$ and fixed $\ell(I)$, we also get
$$
\lim_{|c(I)|_{\infty }\rightarrow \infty }\| T(f)\chi_{(\lambda I)^{c}}\|_{L^{p}(\mathbb R^{n})}=0.
$$
The last two properties do not hold in general for bounded singular integral operators.


\end{document}